\documentclass[a4paper,10pt]{amsproc}

\usepackage{amsfonts, amsmath, amssymb, graphicx, mathrsfs}
\usepackage[all]{xy}

\newdir{ >}{!/6pt/@{ }*:(1,0.2)@_{>}*:(1,-0.2)@^{>}}

\theoremstyle{plain}

\newtheorem{theorem}{Theorem}[section]

\theoremstyle{definition}
\newtheorem{definition}[theorem]{Definition}

\newtheorem{counter example}[theorem]{Counter Example}

\newtheorem{corollary}[theorem]{Corollary}
\newtheorem{example}[theorem]{Example}
\newtheorem{question}[theorem]{Question}
\numberwithin{equation}{section}

\begin{document}

\title[$Ps$-normal and $Ps$-Tychonoff spaces]{$Ps$-normal and $Ps$-Tychonoff spaces}

%Information for first author:
\author[S. Bag]{Sagarmoy Bag}
\address{Department of Pure Mathematics, University of Calcutta, 35, Ballygunge Circular Road, Kolkata 700019, West Bengal, India}
\email{sagarmoy.bag01@gmail.com}
%Information for second author (if needed):
\author[R.C.Manna]{Ram Chandra Manna}
\address{Ramakrishna Mission Vidyamandir, Belur Math, Howrah-711202, West Bengal, India}
\email{mannaramchandra8@gmail.com}
%Information for third author (if needed):
\author[S. K. Patra]{Sourav Kanti Patra}
\address{Ramakrishna Mission Vidyamandir, Belur Math, Howrah-711202, West Bengal, India}
\email{souravkantipatra@gmail.com}
\thanks{The first author thanks the NBHM, Mumbai-400 001, India, for financial support. The second author thanks to Government of West Bengal, India, for financial support.}

%    General info
%%%%%%%%%%%%%%%%%%%%%%%%%%%%%%%%%%%%%%%%%%%%%%%%%%%
%                                                                                                                           %
%         Please use the current 2010 Mathematics Subject Classification:             %
%         http://www.ams.org/mathscinet/msc/                                                        %
%         http://www.zentralblatt-math.org/msc/en/                                                 %
%%%%%%%%%%%%%%%%%%%%%%%%%%%%%%%%%%%%%%%%%%%%%%%%%%%

\large
\thanks {}

\maketitle
\begin{abstract}
A space $X$ is called $Ps$-normal($Ps$-Tychonoff) space if there exists a normal(Tychonoff) space $Y$ and a bijection $f: X\mapsto Y$ such that $f\lvert_K:K\mapsto f(K)$ is homeomorphism for any pseudocompact subset $K$ of $X$. We establish a few relations between $C$-normal, $CC$-normal, $L$-normal, $C$-Tychonoff, $CC$-Tychonoff spaces with $Ps$-normal and $Ps$-Tychonoff spaces. 
\end{abstract}

\section{Introduction}

A topological space $X$ is called Pseudocompact if every real valued continuous function on $X$ is bounded.

\begin{definition}	
	A space $X$ is called $C$-normal(resp. $CC$-normal, $L$-normal) space if there exists a normal space $Y$ and a bijection $f: X\mapsto Y$ such that $f\lvert_K:K\mapsto f(K)$ is homeomorphism for any compact (resp. countably compact, Lindelof) subspace $K$ of $X$.
\end{definition}

A. V. Arhangel'skii introduced the notion of $C$-normal spaces while presenting a talk in a seminar held in Mathematics Department, King Abdulaziz University at Jeddah, Saudi Arabia in 2012. $C$-normality is a generalization of normality. Furthermore the authors in \cite{KA2017} initiated $CC$-normal spaces as another generalization of normal spaces. Incidentally every $CC$-normal space is $C$-normal. In the first part of this paper we offer a new generalization of normal spaces namely $Ps$-normal space. We show that each $Ps$-normal space is $CC$-normal. We give an example of a $Ps$-normal space which is not normal [Example 2.3] and a $CC$-normal space which is not $Ps$-normal [Example 2.6].

\begin{definition}	
	A space $X$ is called $C$-Tychonoff( $L$-Tychonoff) space if there exists a Tychonoff space $Y$ and a bijection $f: X\mapsto Y$ such that $f\lvert_K:K\mapsto f(K)$ is homeomorphism for any compact ( Lindelof) subspace $K$ of $X$.
\end{definition}

We define $Ps$-Tychonoff spaces in the second part of this paper. We show that $Ps$-normality and $Ps$-Tychonoffness are independent to each other. $Ps$-normality and $Ps$-Tychonoffness are both topological property and also additive. $Ps$-normality is not heriditary property but $Ps$-Tychonoff property is heriditary.

In last section we make some remarks on $C$-normal and $C$-Tychonoff spaces.

\section{Ps-normal space}

\begin{theorem}
	Every $Ps$-normal space is $CC$-normal.
\end{theorem}

\begin{proof}
	Let $X$ be $Ps$-normal space and $C$ be a countably compact subspace of $X$. $X$ is $Ps$-normal, then there exists a normal space $Y$ and a bijective map $f:X\mapsto Y$ such that $f\lvert_K:K\mapsto f(K)$ is homeomorphism for each pseudocompact subset $K$ of $X$. Since $C$ is countably compact then it is pseudocompact. Therefore $f\lvert_C:C\mapsto f(C)$ is homeomorphism. Hence $X$ is $CC$-normal space.	
\end{proof}	

Also every $CC$-normal space is $C$-normal. It is easy to check that every normal space is $Ps$-normal. Hence we have
\\
\\
Normal $\Rightarrow$ $Ps$-normal $\Rightarrow$ $CC$-normal $\Rightarrow$ $C$-normal

\begin{theorem}
	\begin{itemize}
		\item[a)] If $X$ is $C$-normal and any pseudocompact subspace of $X$ is contained in a compact subspace of $X$, then $X$ is $Ps$-normal.
		
		\item[b)]  If $X$ is $CC$-normal and any pseudocompact subspace of $X$ is contained in a countably compact subspace of $X$, then $X$ is $Ps$-compact.
	\end{itemize}
\end{theorem}

\begin{proof}
	We prove $(a)$ only. $(b)$ can be prove analogously.
	
	Let $Y$ be a normal space and $f:X\mapsto Y$ be bijective mapping such that $f\lvert_K:K\mapsto f(K)$ is a homeomorphism for every compact subspace $K$ of $X$.
	
	Let $K$ be a pseudocompact subspace of $X$. By hypothesis there exists a compact subspace $A$ of $X$ such that $K\subseteq A$. Since $X$ is $C$-normal, then $f\lvert_A:A\mapsto f(A)$ is homeomorphism and hence $f\lvert_K:K\mapsto f(K)$ is homeomorphism. Therefore $X$ is $Ps$-compact.
\end{proof}	

The following example gives us a $Ps$-normal space which is not normal.

\begin{example}
	The square of Sorgenfrey line $\mathbb{R}_l\times \mathbb{R}_l$ is $C$-normal from Theorem 1.2 in \cite{AK2017} but it is not normal. We want to show $\mathbb{R}_l\times \mathbb{R}_l$ is $Ps$-normal. Let $S$ be a pseudocompact subset of $\mathbb{R}_l\times \mathbb{R}_l$ and $S_1,S_2$ be the first and second projections of $S$ respectively. Then $S_1,S_2$ are pseudocompact. $\mathbb{R}_l$ is hereditarily normal. Therefore $S_1, S_2$ are normal, $T_1$ and pseudocompact. This implies that $S_1, S_2$ are countably compact. Again $\mathbb{R}_l$ is hereditarily Lindel$\ddot{o}$f. Thus $S_1, S_2$ are countably compact and Lindel$\ddot{o}$f. Therefore $S_1, S_2$ are compact in $\mathbb{R}_l\times \mathbb{R}_l$. Hence $S_1\times S_2$ is compact in $\mathbb{R}_l\times \mathbb{R}_l$. Also $S\subseteq S_1\times S_2$. Hence by Theorem 2.2(a), $\mathbb{R}_l\times \mathbb{R}_l$ is $Ps$-normal.
\end{example}

\begin{theorem}
	$Ps$-normality is a topological property.
\end{theorem}

\begin{proof}
	Let $X$ be a $Ps$-normal and $X$ be homeomorphic to $Z$. Let $Y$ be a normal space and $f:X\mapsto Y$ be bijective mapping such that $f\lvert_S:S\mapsto f(S)$	is homeomorphism for any pseudocompact subspace $K$ of $X$.
	Let $g:X\mapsto Z$ be the homeomorphism. Then $f\circ g^{-1}:Z\mapsto Y$ is the required map.
\end{proof}	

\begin{theorem}
	If $X$ is pseudocompact but not normal, then $X$ is not $Ps$-normal.
\end{theorem}

\begin{proof}
	Let $X$ be pseudocompact but not normal. If possible let there exist a normal space $X$ and a bijective mapping $f: X\mapsto Y$ such that $f\lvert_K :K\mapsto f(K)$ is a homeomorphism for any pseudocompact subspace $K$ of $X$. In particular $f:X\mapsto Y$ is a homeomorphism. This is a contradiction as $Y$ is normal and $X$ is not normal. Hence $X$ is not $Ps$-normal.
\end{proof}

We have seen that every $Ps$-normal space is $CC$-normal. The following example shows that the converse is not true.

\begin{example}
	Consider $(\mathbb{R},\tau_c)$, where $\tau_c$ is the co-countable topology on set of all real numbers $\mathbb{R}$. In Example 2.6, \cite{KA2017} it is shown that $(\mathbb{R},\tau_c)$ is $CC$-compact. $(\mathbb{R},\tau_c)$ is pseudocompact as each real valued continuous function on it is constant but it is not normal. Hence by Theorem 2.5, $(\mathbb{R},\tau_c)$ is not $Ps$-normal.

    More generally every irreducible space (i.e a space which has no pair of disjoint nonempty open sets) which is not normal is not $Ps$-normal.	
\end{example}

\begin{theorem}
	Let $X=\oplus_{\alpha\in\Lambda}X_\alpha$. $C\subseteq X$ is psedocompact if and only if $\Lambda_{\circ}=\{\alpha\in \Lambda : C\cap X_\alpha\neq \phi\}$ is finite and $C\cap X_\alpha$ is pseudocomapct for each $\alpha\in \Lambda$.
\end{theorem}

\begin{proof}
	We shall show that if $C\subseteq X$ is pseudocompact then $\Lambda_\circ$ is finite. Others parts are trivial.
	
	Let $C\subseteq X$ be pseudocompact. If possible $\Lambda_\circ$ be infinite. Let $\{\alpha_n:n\in \mathbb{N}\}$ be a countably infinite subset of $\Lambda_\circ$. Define a function $f:C
	\mapsto \mathbb{R}$ by $f(x)=n$ for $x\in X_{\alpha_n}$ and zero otherwise. Then $f\lvert_{C}:C\mapsto \mathbb{R}$ is continuous and unbounded function on $C$. This implies that $C$ is not pseudocompact, a contradiction.
\end{proof}

\begin{theorem}
	$Ps$-normality is an additive property.
\end{theorem}

\begin{proof}
	Let $X=\oplus_{\alpha\in\Lambda}X_\alpha$ and each $X_\alpha$ is pseudocompact. Then for each $\alpha\in \Lambda$ there exists a normal space $Y_\alpha$ and a bijective mapping $f_\alpha :X_\alpha\mapsto Y_\alpha$ such that $f_\alpha\lvert_{C_\alpha}:C_\alpha\mapsto f_\alpha (C_\alpha)$ is homeomorphism for each pseudocompact subspace $C_\alpha$ of $X_\alpha$. Then $Y=\oplus_{\alpha\in\Lambda}Y_\alpha$ is a normal space. Consider the function $f=\oplus_{\alpha\in\Lambda}f_\alpha :X\mapsto Y$. Let $C\subseteq X$ be pseudocompact subspace of $X$. Then by Theorem 2.7, $\Lambda_{\circ}=\{\alpha\in \Lambda : C\cap X_\alpha\neq \phi\}$ is finite and $C\cap X_\alpha$ is pseudocomapct for each $\alpha\in \Lambda$. Then $f\lvert_C: C\mapsto f(C)$ is a homeomorphism. Hence $X$ is $Ps$-normal.
\end{proof}

A space $X$ is called Frechet if for any subset $B$ of $X$ and $x\in \overline{B}$ there exist a sequence $\{b_n\}$ of points of $B$ such that $b_n\rightarrow x$.

\begin{theorem}
	If $X$ is $Ps$-normal and Frechet, $f:X\mapsto Y$ is witness of $Ps$-normality, then $f$ is a continuous mapping.
\end{theorem}

\begin{proof}
	It is sufficient to show that for any $A\subseteq X, f(\overline{A})\subseteq \overline{f(A)}$.
	
	Let $A\subseteq X$ and $y\in f(\overline{A})$. Let $x$ be the unique point in $\overline{A}$ such that $f(x)=y$. Since $X$ is Frechet, then there exists a sequence $\{a_n\}$ in $A$ such that $x$ is a limit of $\{a_n\}$. Let $K=\{a_n:n\in \mathbb{N}\}\cup \{x\}$. Then $K$ is compact and therefore pseudocompact. Then $f\lvert_K: K\mapsto f(K)$ is homeomorphism as $X$ is $Ps$-normal.
	
	Let $U$ be any open set in $Y$. Then $U\cap f(K)$ is open set in $f(K)$ containing $y$. Now $f(\{a_n:n\in\mathbb{N})\subseteq f(K)\cap f(A)$ and $U\cap f(K)\neq \phi$. Then $U\cap f(A)\neq \phi$. Therefore $y\in \overline{f(A)}$. Hence $f(\overline{A})\subseteq \overline{f(A)}$.
\end{proof}

\section{$Ps$-Tychonoff spaces}

\begin{definition}
	A space $X$ is called $Ps$-Tychonoff($C$-Tychonoff) space if there exists a Tychonoff space $Y$ and a bijective mapping $f:X\mapsto Y$ such that $f\lvert_C: C\mapsto f(C)$ is homeomorphism for every Pseudocompact(compact) subset $C$ of $X$.
\end{definition}

On using the arguments of Theorems 2.1, 2.5, 2.9, 2.2, 2.4, 2.8 closely, we get the following properties respectively.

Every Tychonoff space is $Ps$-Tychonoff space.

\begin{theorem}
	Every $Ps$-Tychonoff space is $C$-Tychonoff.
\end{theorem}

\begin{theorem}
	If $X$ is Pseudocompact but not Tychonoff, then $X$ is not $Ps$-normal.
\end{theorem}

Therefore $\mathbb{R}$ with co-finite topology is not $Ps$-normal.

\begin{theorem}
If $X$ is $Ps$-Tychonoff and Frechet,  $f:X\mapsto Y$ is witness of $Ps$-Tychonoffness, then $f$ is continuous.
\end{theorem}

\begin{theorem}
	Let $X$ be an $L$-Tychonoff space. If each Pseudocompact subset in $X$ is contained in a Lindeloff subset, then $X$ is $Ps$-Tychonoff.
\end{theorem}

\begin{theorem}
	$Ps$-Tychonoffness is a topological property.
\end{theorem}

\begin{theorem}
	$Ps$-Tychonoffness is additive property.
\end{theorem}

\begin{theorem}
	If $X$ is $T_1$ space and only Pseudocompact subset of $X$ is finite, then $X$ is $Ps$-Tychonoff.
\end{theorem}

\begin{proof}
	Let $Y=X$ and $Y$ be discrete topological space. Let $I: X\mapsto X$ be the identity mapping. Let $K$ be pseudocompact subset of $X$. Then by the hypothesis $K$ is finite. Therefore $K$ is finite and $T_1$. therefore $K$ is a discrete topological space. Thus $I: K\mapsto I(K)$ is homeomorphism. Henece $X$ is $Ps$-Tychonoff.
\end{proof}	

\begin{example}
	Let $X=\mathbb{R}\cup \{ i,j\}$, where $i\notin \mathbb{R}$ and $j\notin \mathbb{R}$. Define a topology on $X$ as follows: Each point of $\mathbb{R}$ is isolated. A basic open neighbourhood of $i$ is a subset $U$ of X containing $i$ and $X\setminus U$ is countable. Also a basic neighbourhood of $j$ is a subset $V$ of $X$ containing $j$ and $X\setminus V$ is countable. Clearly $X$ is $T_1$ space. But $X$ is not $T_2$. Therefore $X$ is not Tychonoff.
	
	We want to show that the Pseudocompact subsets of $X$ are finite. Let $A$ be infinite subset of $X$. If $i, j\notin A$, then $X$ is infinite discrete and so it is not Pseudocompact.
	
	If $i, j\in A$. Let $B=\{a_n:n\in \mathbf{N}\}$ be countably infinite subset of $A$ and $i, j\notin B$. Define $f: A\mapsto \mathbb{R}$ as follows: $f(a_n)=n$ for each $n\in \mathbb{N}$ and zero otherwise. Each point of $A\setminus \{i,j\}$ is isolated and therefore $f$ is continuous on $A\setminus \{i,j\}$. $f(i)=0$. Let $U=A\setminus B$. Then $U$ is an open neighbourhood of $i$ and $f(U)=\{0\}$. Therefore $f$ is continuous at $i$. Similarly $f$ is continuous at $j$. Therefore $f$ is an unbounded continuous function on $A$. Therefore $A$ is not Pseudocompact.
	
	Similarly we can show that if either $i\in A$ or $j\in A$, then $A$ is not Pseudocompact.
	
	Therefore $X$ is $T_1$ and each Pseudocompact subset is finite. Therefore $X$ is $Ps$-Tychonoff.
\end{example}	

\begin{theorem}
	Every $Ps$-Tychonoff Frechet Lindeloff space is $Ps$-normal.
\end{theorem}

\begin{proof}
	Let $X$ be $Ps$-Tychonoff Frechet Lindeloff space. Then there exists a Tychonoff space $Y$ and a bijective mapping $f: X\mapsto Y$ such that $f\lvert_K: K\mapsto f(K)$ is homeomorphism for any Pseudocompact subset $K$ of $X$. By Theorem 3.4, $f$ is continuous. Since continuous image of Lindeloff space is Lindeloff, therefore $Y$ is Lindeloff. Thus $Y$ is regular Lindeloff, this implies that $Y$ is normal. Hence $X$ is $Ps$-normal.
\end{proof}	

We now show that $Ps$-normality and $Ps$-Tychonoffness are independent properties. We reproduce the following two examples from \cite{AK2018} for our purpose.

\begin{example}
	Let $\tau=\{\emptyset, \mathbb{R}\}\cup \{(x,\infty): a \in \mathbb{R}\}$ be the topology on $\mathbb{R}$. Then $(\mathbb{R},\tau)$ is a normal space and therefore it is $Ps$-normal space. From Example 4,\cite{AK2018}, $(\mathbb{R},\tau)$ is not a $C$-Tychonoff space. Then by Theorem 3.2 it is not a $Ps$-Tychonoff space.
\end{example}

\begin{example}
	Let $G=\prod_{\alpha\in \omega_1}D_\alpha$, where $D_\alpha =\{0,1\}$ for each $\alpha$ and $\omega_1$ is the first uncountable ordinal. Let $H$ be the set of all points of $G$ with atmost contably many non-zero co-ordinates. Let $M=G\times H$. By Example 4 in \cite{AK2018} $M$ is a Tychonoff space and not $C$-normal. Therefore $M$ is a $Ps$-Tychonoff space. But every $Ps$-normal space is $C$-normal and $M$ is not $C$-normal, we have that $M$ is not $Ps$-normal.
\end{example}

\begin{theorem}
	Each subset of $Ps$-Tychonoff space is $Ps$-Tychonoff..
\end{theorem}

\begin{proof}
	Let $X$ be a $Ps$-Tychonoff space and $A$ be a subspace of $X$. Now there exists a Tychonoff space $Y$ and a bijective map $f:X\mapsto Y$ such that for every Pseudocompact subset $C$ of $X$, $f\lvert_C: C\mapsto f(C)$ is homeomorphism.
	
	Let $K$ be a Pseudocompact subset of $A$. Then $K$ is also a Pseudocompact subset of $X$. Then $f\lvert_K:K\mapsto f(K)$ is homeomorphism. Therefore $f\lvert_A$ is the witness of $Ps$-Tychonoffness of $A$.
\end{proof}	

Let $X$ be any topological space. Lat $X'=X\times \{1\}$. Late $A(X)=X\cup X'$. For $x\in X$ we denote $<x,1>$ as $x'$ and for $B\subseteq X$ let $B'=\{x':x\in B\}=B\times \{1\}$. Let $\{\beta (x): x\in X\}\cup \{\beta (x'):x'\in X'\}$ be neighbourhood system for some topology $\tau$ on $A(X)$, where for $x\in X, \beta (x)=\{ U\cup (U'\setminus\{x'\}): U$ is an open set in $X$ containing $x\}$ and for $x'\in X', \beta (x')=\{\{x'\}\}$. $(A(X),\tau)$ is called Alexandroff Duplicate of $X$.

The following problems are open.

\begin{question}
	Is $A(X)$ $Ps$-normal, whwn $X$ is $Ps$-normal?
\end{question}	

\begin{question}
	Is $A(X)$ $Ps$-Tychonoff, when $X$ is $Ps$-Tychonoff?
\end{question}

\section{Some special results}

\begin{theorem}
If $X$ is $C$-Tychonoff and Frechet, then for any compact set $A$ in $X$ and $x\notin A$, there exists a continuous function $h:X\mapsto [0,1]$ such that $h(x)=0$ and $h(A)=\{1\}$.
\end{theorem}

\begin{proof}
Let $f:X\mapsto Y$ be witness of $C$-Tychonoff. Since $X$ is $C$-compact and Frechet, then by Theorem 5\cite{AK2018}, $f$ is continuous. Since $f$ is continuous and $A$ is compact, then f(A) is compact in $Y$. $Y$ is Tychonoff and so it is $T_2$. Therefore $f(A)$ is closed in $Y$ and also $f(x)\in f(A)$. Since $Y$ is Tychonoff, therefore there exists a continuous function $g:Y\mapsto [0,1]$ such that $g(f(x))=0$ and $g(f(A))=\{1\}$. Then $h=g\circ f:X\mapsto [0,1]$ is a continuous function such that $h(x)=0$ and $h(A)=\{1\}$. This completes the proof.
\end{proof}

We give an alternative proof of the following result.

\begin{corollary}[Corollary 2\cite{AK2018}]
Every $C$-Tychonoff Frechet space is Urysonh.
\end{corollary}

\begin{proof}
Let $X$ be $C$-Tychonoff Frechet space and contains at least two points. Let $a, b$ be two distinct points of $X$. Then $\{b\}$ is a compact set and $a\notin \{b\}$. Then there exists a continuous functions$h:X\mapsto [0,1]$ such that $f(a)=0$ and $f(b)=1$. Let $U=\{x\in X: f(x)\leq \frac{1}{3}\}$ and $V=\{x\in X: f(x)\geq \frac{1}{2}\}$. Then $U, V$ are disjoint closed neighbourhoods of $a,b$ respectively. Hence $X$ is Urysonh.
\end{proof}

\begin{theorem}
Let $X$ be a $C$-normal Frechet space. Let $f:X\mapsto Y$ be witness of $C$-normality of $X$. In addition if $Y$ is $T_2$ space, then for any two disjoint compact sets $H, K$ in $X$, there exists a continuous function $h:X\mapsto [0,1]$ such that $h(H)=\{0\}$ and $h(K)=\{1\}$.
\end{theorem}

\begin{proof}
$A$ is $C$-normal and Frechet, then $f$ is continuous. Then $f(H), f(K)$ are disjoint compact sets in the $T_2$ space $Y$. By normality of $Y$, there exists a continuous function $g:Y\mapsto [0,1]$ such that $g(f(H))=\{0\}$ and $g(f(K))=\{1\}$. Then $h=g\circ f:X\mapsto [0,1]$ is a continuous function such that $h(H)=\{0\}$ and $h(K)=\{1\}$. This completes the proof.
\end{proof}

\textbf{Acknowledgments:} The authors would like to thank to Professor Sudip Kumar Acharyya and Professor Asit Baran Raha for their valuable suggestions.

\bibliographystyle{plain}

\begin{thebibliography}{6}
	
	\bibitem{AK2018}
	S. AlZahrani,
	$C$-Tychonoff and $L$-Tychonoff topological spaces, European Journal of Pure and Applied Mathematics, {\textbf 11} (2018), 882--892.
	
	
	\bibitem{AK2017}
	S. AlZahrani, L. Kalantan,
	$C$-normal topological property, Filomat, {\textbf 31} (2017), 407--411.
	
	\bibitem{H1948}
	E. Hewitt, Rings of real valued continuous functions. I, Trans. Amer. Math. Soc., 64(1948), 45--99.
	
	\bibitem{KA2017}
	L. Kalantan, M. Alhomieyed,
	$CC$-normal topological spaces, Turkish Journal of Mathematics, {\textbf 41} (2017), 749--755.


	
\end{thebibliography}

\end{document}